\documentclass[11pt]{article}
\usepackage[a4paper,hmargin=1in,vmargin=1.25in]{geometry}
\usepackage{amsmath,amsfonts,amssymb,authblk,url,cite}%amsthm}
\usepackage{booktabs}
\usepackage[toc,page]{appendix}
\usepackage{longtable}
\usepackage{blkarray}
\usepackage{xcolor}
\usepackage{amsmath,amsfonts,amssymb,amsthm,url}%,cite}
\usepackage{authblk}
\renewcommand\footnotemark{}
\usepackage{enumitem}

\usepackage{tikz}
\usetikzlibrary{shapes,arrows}
\usetikzlibrary{backgrounds}

\usepackage{float}
\def\Real{\mathbb{R}} % The real numbers
\def\Par{\mathcal{P}} % Power set
\def\Acc{\Gamma} % Access structure (global)
\def\Poly{\mathcal{S}} % A polymatroid
\def\Mat{\mathcal{M}} % A matroid
\def\field{\mathbb{F}}
\def\kos{\mathbb{K}}% A finite field
\def\sss{\smallsetminus}

\newcommand*{\lon}{
       \mskip1mu
        \relax
        {:}
        \mskip1mu
        \relax
} 

\theoremstyle{plain}
\newtheorem{proposition}{Proposition}[section]
\newtheorem{theorem}[proposition]{Theorem}
\newtheorem{corollary}[proposition]{Corollary}
\newtheorem{lemma}[proposition]{Lemma}

\theoremstyle{definition}
\newtheorem{definition}[proposition]{Definition}

\newtheorem{remark}[proposition]{Remark}
\newtheorem{lpp}[proposition]{Linear Programming Problem}

\title{Common Information, Matroid Representation, and Secret Sharing for Matroid Ports}
\author[1]{Michael Bamiloshin}
\affil[1]{Universitat Rovira i Virgili, Tarragona, Catalonia, Spain}
\author[2]{Aner Ben-Efraim} 
\affil[2]{Ariel University, Ariel, Israel}
\author[1]{Oriol Farr\`as}
\author[3]{Carles Padr\'o}
\affil[3]{Universitat Polit\`ecnica de Catalunya, Barcelona, Spain\\
\texttt{michael.bamiloshin@urv.cat, anermosh@post.bgu.ac.il, oriol.farras@urv.cat, carles.padro@upc.edu}}

\begin{document}
%%%%%%%%%%%%%%%%%%%%

\maketitle

\begin{abstract}
Linear information and rank inequalities as,
for instance, Ingleton inequality, 
are useful tools in information theory and matroid theory.
Even though many such inequalities have
been found, it seems that most of them remain undiscovered.
Improved results have been obtained
in recent works by using the properties from which they are derived
instead of the inequalities themselves.
We apply here this strategy to 
the classification of matroids according to their representations
and to the search for bounds on secret sharing for matroid ports.

\noindent 
\textbf{Key words.}
Matroid representation, Secret sharing, Information inequalities, Common information, Linear programming.
%\keywords{Matroid representation, Secret sharing, Information inequalities, Common information, Linear programming}
%\subclass{68P30, 52B40, 94A62, 94A60}
%94A60 - Cryptography
%68P30 - Coding and information theory
%94A62 - Authentication and secret sharing
%52B40 - Matroids
\end{abstract}
\makeatletter{\renewcommand*{\@makefnmark}{}
	\footnotetext{The first and third authors were supported by the grant 2017 SGR 705 from the Government of Catalonia and grant RTI2018-095094-B-C21 “CONSENT” from the Spanish Government. Also, the first author has received funding from the European Union's Horizon 2020 research and innovation programme under the Marie Sk\l{}odowska-Curie grant agreement No. 713679 and from the Universitat Rovira i Virgili. The third author was supported by ISF grant 152/17. The fourth author was supported by the Spanish Government through grant MTM2016-77213-R.}\makeatother}

%%%%%%%%%%%%%%%%%%%%%%%%%%%%%%%%%%%%%%%%%%%%%%%%%%%%%
%%%%%%%%%%%%%%%%%%% END TITLE ABSTRACT, ETC   %%%%%%%
%%%%%%%%%%%%%%%%%%%%%%%%%%%%%%%%%%%%%%%%%%%%%%%%%%%%%

%%%%%%%%%%%%%%%
%Section 1
%%%%%%%%%%%%%%%

%%%%%%%%%%%%%%%%%%%%%%%%%%%%%%%%%%%%%%%%%%%%%%%%%%%%%%%%%
\section{Introduction}
%%%%%%%%%%%%%%%%%%%%%%%%%%%%%%%%%%%%%%%%%%%%%%%%%%%%%%%%%%%%

Some of the concepts appearing next
are defined in Section~\ref{sec:MatrandPolys}.
The reader is referred to the books~\cite{Oxl92,Wel76} on matroid theory 
and~\cite{Yeu08} on information theory, and the surveys~\cite{Bei11,Pad12} on secret sharing
for additional information about these topics.

%%%%%%%%%%%%%%%%%%%%%%%%%%%%%%%%%%%%%%%
\subsection{Matroid Representation}
\label{pt:MatRep}
%%%%%%%%%%%%%%%%%%%%%%%%%%%%%%%%%%%%%

Relevant applications in information theory,
especially in secret sharing and network coding,
brought to light the class of \emph{entropic} matroids,
which contains the well-known class of linear matroids. 

An \emph{entropic vector} is formed by
the joint Shannon entropies of all subsets 
of a finite set of discrete random variables.
Every entropic vector is the rank function
of a polymatroid.
A polymatroid is \emph{entropic} if its
rank function is a multiple of an entropic vector.
Limits of entropic polymatroids
are called \emph{almost entropic}.
Both representation by partitions~\cite{Matus99}
and by almost affine codes~\cite{SiAs98}
are characterizations of entropic matroids.

In the same way that linear matroids
are defined from configurations of vectors in a vector space,
configurations of vector subspaces determine
\emph{linear polymatroids}.
A \emph{folded linear} matroid is such that 
some multiple of its rank function 
corresponds to a linear polymatroid.
Folded linear matroids have been called 
\emph{multilinear} or \emph{multilinearly representable}
in the literature. 
Since no multilinear algebra is involved, 
that terminology may be misleading.
The name proposed here is motivated by the analogy
with folded Reed-Solomon codes.

It is well known that linear polymatroids and,
consequently, folded linear matroids are entropic.
Franti\v{s}ek Mat\'u\v{s}~\cite{Matus17} recently proved
that algebraic matroids are almost entropic.

Figure~\ref{classes}, an update of the corresponding diagram in~\cite{Matus18},
illustrates the current knowledge about the connections between
the aforementioned classes of matroids.
A detailed explanation is given in Section~\ref{sec:MatrandPolys}. 
There is a number of tools to deal with that classification.
Among them, linear information and rank inequalities are especially useful.
\emph{Linear information inequalities}, such as Zhang--Yeung inequality~\cite{ZhYe98},
are the linear inequalities that are satisfied
by the rank function of every entropic polymatroid.
The ones that, like Ingleton inequality~\cite{Ing71}, 
are satisfied by the rank function of every linear polymatroid
are called \emph{linear rank inequalities}.

\begin{figure}
\centering
\begin{tikzpicture}[scale=0.9,framed,background rectangle, squarednode/.style={rectangle, draw=red!60, fill=red!5, very thick, minimum size=5mm},
]
\draw[black, very thick] (2.5,0) node{entropic};
\draw[black, very thick] (4.6,0.2) node{almost};
\draw[black, very thick] (4.6,-0.2) node{entropic};
\draw[blue, very thick] (-0.2,0) ellipse (39 mm and 33mm); %entropic
\draw[blue, very thick] (-0.4,0) ellipse (61 mm and 36mm); %almost entropic
\draw[blue, very thick] (-3.7,0) ellipse (26 mm and 18mm); %algebraic
\draw[black, very thick] (-5.2,0) node{algebraic};
\draw[blue, very thick] (-1.2,0) ellipse (27mm and 19mm); %multilinear
\draw (0.1,0) node{folded linear};
\draw (0.1,-2.5) node{?};
\draw[blue, very thick] (-2.5,0) ellipse (12mm and 10mm); %linear
\draw (-2.5,0) node{linear};
\end{tikzpicture}
\caption{A classification of matroids. Discussed in Section~\ref{sec:MatrandPolys}.}
\label{classes}
\end{figure}

Ingleton inequality was used to prove the existence of an infinite number of
excluded minors for the class of matroids 
that are linear over any given infinite field~\cite{MNW09}.
That result has been extended to the class of 
folded linear matroids over any given field and,
by using Zhang--Yeung inequality instead of Ingleton inequality, 
to the classes of almost entropic matroids and 
algebraic matroids~\cite{Matus18}. 

%%%%%%%%%%%%%%%%%%%%%%%%%%%%%%%%%%%%%%%%%%%%%%%%
\subsection{Common Information}
\label{pt:intCI}
%%%%%%%%%%%%%%%%%%%%%%%%%%%%%%%%%%%%%%%%%%%%%%%

Besides Ingleton and Zhang--Yeung inequalities,
many other linear information and rank inequalities have
been found~\cite{DFZ06,DFZ09,DFZ11,Kin11,MMRV02,Mat07-3}.
Nevertheless, only a few techniques to derive 
such inequalities are known, and it appears 
that many more inequalities remain unknown.

Linear information and rank inequalities
are fundamental in the linear programming technique
that has been used to find bounds on the information ratio of
secret sharing schemes~\cite{BLP08,BeOr11,MPY16,Met11,PVY13}
and on the achievable rates in network coding~\cite{DFZ07,TCG17,Yeu08}. 
An improvement to that technique
has been recently proposed~\cite{FKMP20}.
Specifically, instead of known inequalities, the properties from which most
linear information and rank inequalities are derived are used as constraints. 
The notion of \emph{common information} of two random variables
is at the core of most of those properties.
Most of the known linear information inequalities are obtained
from the concept of \emph{AK-common information},
derived from {Ahlswede--K\"{o}rner lemma}~\cite{AhKo77,AhKo06,CsKo81},
or from the \emph{copy lemma}~\cite{DFZ06,DFZ11}.
According to~\cite{DFZ09}, all linear rank inequalities
that were known in 2009 were derived from the \emph{common information property}
and, to the best of our knowledge, that is still the case nowadays. 
Nevertheless, some restricted linear rank inequalities
have been presented since then.
Namely, \emph{characteristic-dependent}
inequalities~\cite{DFZ15,PeSa19}, 
which are satisfied by
all polymatroids that are linearly
representable over fields of a given characteristic.

Several new lower bounds on the information ratio of secret  
sharing schemes have been obtained by using
that improved linear programming technique~\cite{FKMP20}.
For instance, by using the common information property, the exact
values of the optimal information ratios of \emph{linear} secret sharing schemes
for \emph{all} access structures on five players and \emph{all}
graph access structures on six players have been determined,
concluding the projects undertaken in~\cite{Dij95,JaMa96}
when restricted to linear schemes.
Moreover, some of the existing lower bounds for general
(that is, non-linear) secret sharing schemes for
those and other access structures have been improved
by using the AK-common information.
The analogous application of the copy lemma
has been described in~\cite{GuRo2019}.

On the negative side, the application of that technique
is currently limited to solving linear programming problems
that provide bounds for particular cases.
Moreover, because of the huge number of
variables and  constraints, only problems
with small size can be solved.
In contrast, several general results,
such as the best known general lower bound for secret sharing~\cite{Csi97},
have been obtained from the simpler technique involving
only Shannon inequalities.

The search for new techniques
to derive linear rank and information inequalities and
further improve the aforementioned linear programming technique 
is worth undertaking.
For example, the common information
property is solely based on the intersection of vector subspaces.
It is possible that a deeper use of linear algebra, as in the search
for characteristic-dependent linear rank inequalities~\cite{DFZ15,PeSa19},
will provide some results.

\subsection{Secret Sharing for Matroid Ports}

A perfect secret sharing scheme is \emph{ideal}
if all shares have the same size as the secret value,
which is the smallest possible. 
The entropic vector given by the random variables
defining an ideal scheme determines an entropic matroid~\cite{BrDa91,Matus99}.
The access structure is a port of that matroid~\cite{BrDa91,MaPa10}.
As a consequence, the access structures of ideal secret sharing
schemes are precisely the ports of entropic matroids,
while the ports of folded linear matroids
coincide with the access structures of 
ideal \emph{linear} secret sharing schemes. 

The optimal information ratio of 
secret sharing schemes for the ports of a matroid
measures in some way how far it is from being entropic.
This parameter has been studied for the
Vamos matroid~\cite{BeLi08,BLP08,FKMP20,GuRo2019,MaPa10,Met11},
the first known example of a non-entropic matroid~\cite{Sey92},
and also for other non-entropic matroids~\cite{FKMP20,PVY13}.
For the ports of the Vamos matroid,
the application of the linear programming technique
with the common information property 
yielded the exact value of the optimal 
information ratio of linear secret sharing schemes~\cite{FKMP20}.
Moreover, G\"urpinar and Romashchenko~\cite{GuRo2019}
recently obtained the current best lower bound for the general case
by using that technique with the copy lemma.

%%%%%%%%%%%%%%%%%%%%%%%%%%%%%%%%%%%%%%%%%
\subsection{Our Results}
%%%%%%%%%%%%%%%%%%%%%%%%%%%%%%%%%%%%%%%%%

We investigate the application of the improved linear
programming technique introduced in~\cite{FKMP20}
to the classification of matroids
according to the different representations 
discussed in Section~\ref{pt:MatRep}.
First, we prove in Theorem~\ref{st:Ingspar}
an interesting consequence of the results 
by Nelson and van der Pol~\cite{NP18}. 
Namely, every almost entropic sparse paving matroid
must satisfy Ingleton inequality.
Second, we present an almost complete classification
of the matroids on eight points.
Our starting point is the paper by Mayhew and Royle~\cite{MaRo08}, 
in which the linear matroids on eight points are determined.
Specifically, up to isomorphism, there are exactly $44$ matroids on eight points
that are not linear. All of them are sparse paving matroids.
Exactly $39$ of them do not satisfy Ingleton inequality, 
and hence they are not almost entropic. 
Therefore, there are five sparse paving matroids 
that are not linear but satisfy Ingleton inequality.
We prove in Section~\ref{pt:folded}
that exactly two of them are folded linear matroids.
They are the smallest folded linear matroids that are not linear.  
Those two matroids were known to be algebraic.
Unfortunately, we could not determine
whether or not the other three matroids are  
algebraic or almost entropic.
Some results about matroids on nine points
are presented in Section~\ref{subs:9,10-ptMat}.
Specifically, we found $171$ that
satisfy Ingleton inequality but do not 
have the common information property.
They are among the smallest matroids in that situation.
One of those examples is the tic-tac-toe matroid. 
Those $171$ matroids are not folded linear,
but we could not determine whether or not
they are algebraic or almost entropic.

In addition, by using the improved linear programming technique, 
we find new lower bounds on the information ratio
of secret sharing schemes for several matroid ports.
By combining our bounds for matroids on eight points
with the results in~\cite{NP18},
we present in Theorem~\ref{thm:l&s4nonIC}
lower bounds that apply to 
every sparse paving matroid that
does not satisfy Ingleton inequality.
We found a lower bound on the information
ratio of \emph{linear} secret sharing
schemes for the ports of
the tic-tac-toe matroid and some of the 
aforementioned $171$ related matroids.
Finally, we determined the exact value
of the optimal information ratio of
linear secret sharing schemes for a port
of the tic-tac-toe matroid.

In this work, we used the Gurobi\textsuperscript{TM} optimizer for solving the linear programming problems, and the SageMath matroid package for specific matroid operations. 
%Oriol: code added
The code we used is available at \url{https://github.com/bmilosh/Common-Information-and-Matroid-Ports}.

%%%%%%%%%%%%%%%%%%%%%%%%%%%%%%%%%%%%%%%%
%%%%%%%%%% PRELIMINARIES %%%%%%%%%%%%%%%
%%%%%%%%%%%%%%%%%%%%%%%%%%%%%%%%%%%%%%%%
\section{Preliminaries}\label{sec:MatrandPolys}
%%%%%%%%%%%%%%%%%%%%%%%%%%%%%%%%%%%%%%%%
%%%%%%%%%%%%%%%%%%%%%%%%%%%%%%%%%%%%%%%%

The number of elements of the finite set $X$ is denoted by $|X|$ 
and $\Par(Q)$ denotes the power set of $Q$.
For a positive integer $m$, we notate $[m] = \{1, \ldots, m\}$. 
We use a compact notation for set unions,
that is, we write $XY$ for $X \cup Y$ and $Xy$ for $X \cup \{y\}$.
In addition, we write $X \sss Y$ for the set difference
and $X \sss x$ for $X \sss \{x\}$.
The reader should be aware that a slightly different operation symbol
is used in expressions like 
$M \setminus B$ or $\Gamma \setminus B$
to denote operations that will be defined later in this section, 
namely, \emph{deletion} 
in polymatroids or, respectively, access functions.

%For a finite sequence $S = (S_x)_{x \in Q}$
%with set of indices $Q$ and a subset $X \subseteq Q$,
%the subsequence $(S_x)_{x \in X}$ is denoted by $S_X$.

%%%%%%%%%%%%%%%%%%%%%%%%%%%%%%%%%%%%%%%%%%%%
\subsection{Matroids and Polymatroids}
\label{pt:matbas}
%%%%%%%%%%%%%%%%%%%%%%%%%%%%%%%%%%%%%%%%%%%%

\begin{definition}\label{df:polym1}
Given a finite set $Q$ and a function
$f\colon\mathcal{P}(Q)\rightarrow \Real$, 
the pair $(Q,f)$ is called a \textit{polymatroid} if
the following properties are satisfied for all $X,Y\subseteq Q$.
\begin{description}
	\item[(P1)] $f(\emptyset) = 0$.
	\item[(P2)]{\label{P2}} $f(X) \leq f(Y)$ if $X \subseteq Y$.
	\item[(P3)]{\label{P3}} $f(X \cap Y) + f(X \cup Y) \le f(X) + f(Y)$.
\end{description}
The set $Q$ and the function $f$ are, respectively,
the \textit{ground set} and the \textit{rank function} of the polymatroid.
The rank function of an \emph{integer polymatroid} only takes integer values.
A \emph{matroid} is an integer polymatroid  $(Q,r)$ 
such that $r(X) \le |X|$ for every $X \subseteq Q$.
\end{definition}

Some additional terminology and
properties about matroids are needed.
Let $M = (Q,r)$ be a matroid. 
The \emph{independent sets} of $M$ 
are the sets $X \subseteq Q$ with $r(X) = |X|$.
Every subset of an independent set is independent. 
The \emph{bases} of $M$ are the maximal independent sets,
and the minimal dependent sets are the \emph{circuits}.
All bases have the same number of elements,
which equals $r(Q)$, the \emph{rank of the matroid}.
A set $X \subseteq Q$ is a \emph{flat} of $M$
if $r(Xx) > r(X)$ for every $x \in Q \sss X$. 
The flats with rank $r(Q)-1$ are called \emph{hyperplanes}. 
In addition to the one given in Definition~\ref{df:polym1},
there are other equivalent sets of axioms characterizing matroids
which are stated in terms of the properties
of the independent sets, the circuits, 
the bases, or the hyperplanes.

In a \emph{simple} matroid,
all sets with one or two elements are independent.
A matroid of rank $k$ is \emph{paving}
if the rank of every circuit is either $k$ or $k - 1$.
It is \emph{sparse paving} if,
in addition, all circuits of rank $k - 1$ are flats,
which are called \emph{circuit-hyperplanes}.
The \emph{dual} of $M = (Q,r)$ 
is the matroid $M^* = (Q,r^*)$ with
$r^*(X) = |X| - r(Q) + r(Q \sss X)$
for every $X \subseteq Q$.
Equivalently, $M^*$ is the matroid
on $Q$ whose bases are the complements of the bases of $M$.

We introduce next the operations
that are used to define \emph{minors} of
matroids and polymatroids.
For a polymatroid $M = (Q,f)$ and a set $B \subseteq Q$,
the \emph{deletion $M \setminus B$ of $B$ from $M$}
is the polymatroid $(Q \sss B, \widehat{f})$
with $\widehat{f}(X) = f(X)$ for every $X \subseteq Q \sss B$,
while the \emph{contraction 
$M/B = (Q \sss B, \widetilde{f})$ of $B$ from $M$} is defined by
$\widetilde{f}(X) = f(XB) - f(B)$ for every $X \subseteq Q \sss B$. 
%Notice that we use the symbol $\sss$ for set difference, 
%while we use $\setminus$ for matroid deletion and other related operations.
Every polymatroid that is obtained from $M$ by 
applying deletions and contractions is called a \emph{minor of $M$}.
Finally, observe that minors of matroids are matroids.

Let $S = (S_x)_{x \in Q}$ be a discrete random vector,
that is, a finite sequence of discrete random variables.
For every $X \subseteq Q$, take $h(X) = H(S_X)$,
the Shannon entropy of the 
discrete random variable $S_X = (S_x)_{x \in X}$.
Then $(h(X))_{X \in \Par(Q)}$ is the
\emph{entropic vector} associated to $S$.
Because of the basic properties of Shannon entropy,
every entropic vector is the rank function of a polymatroid~\cite{Fuj78,Fuj78-2}.
A polymatroid is \emph{entropic} if its rank function
is a multiple of an entropic vector.
The closure in $\Real^{\Par(Q)}$ of the set of entropic vectors
is a convex cone~\cite{Yeu08}.
Each element in this convex cone is the rank function
of an \emph{almost entropic} polymatroid.

We introduce next some notation that is motivated by 
this connection between Shannon entropy and polymatroids. 
By analogy with the conditional mutual information,
for a polymatroid $(Q,f)$ and sets $X,Y,Z \subseteq Q$, we write  
\[
f(Y \lon Z | X) =
f(XY) + f(XZ) - f(XYZ) - f(X)
\]
and, in particular,
\(
f(Y \lon Z) = f(Y \lon Z | \emptyset)
= f(Y) + f(Z) - f(YZ)
\)
and
\(
f(Y | X) =f(Y \lon Y | X) = f(XY) - f(X) 
\).

Consider a field $\field$, a vector space $V$ with finite dimension over $\field$
and a collection $(V_x)_{x \in Q}$  of vector subspaces of~$V$.
It is clear from basic linear algebra that the map $f$
defined by 
\(
f(X) = \dim \sum_{x\in X} V_x
\)
for every $X \subseteq Q$
is the rank function of a polymatroid.
Every such polymatroid is said to be 
\emph{linearly representable},
or simply \emph{linear}, over $\field$.
For a positive integer~$k$,
a \emph{$k$-folded $\field$-linear} matroid $(Q,r)$ is such that
the polymatroid $(Q,k r)$ is $\field$-linear.
As we mentioned in the Introduction,
folded linear matroids are also called
\emph{multilinear} or \emph{multilinearly representable}
in the literature.

Suppose now that $\field$ is a finite field
and take the dual vector space $V^*$.
The uniform probability distribution on $V^*$
and the projections $V^* \to V_x^*$ for $x \in Q$
determine a discrete random vector $(S_x)_{x \in Q}$.
Such random vectors are called \emph{linear}.
The entropic vector $h$ associated to $S$
satisfies $h(X) = f(X) \log|\field|$ for every $X \subseteq Q$.
Since every linear polymatroid admits a linear representation
over some finite field~\cite{Rad57},
linear polymatroids and folded linear matroids are entropic.

Consider a field extension $\kos /\field$ 
and a finite collection $(v_x)_{x \in Q}$ of elements in $\kos$.
For every $X \subseteq Q$, let $r(X)$ be the transcendence degree
of the field extension $\field(\{v_x\}_{x \in X})/\field$.
Then $r$ is the rank function of a matroid
$M$ with ground set $Q$.
In this situation, $M$ is \emph{algebraic over $\field$}
and $(v_x)_{x \in Q}$ is an \emph{algebraic representation of $M$}.

Given a positive integer $m$, 
a collection $(A_i)_{i \in [m]}$
of subsets of a finite set $Q$, and $I \subseteq [m]$,
we notate $A_I = \bigcup_{i \in I} A_i$.
A \emph{linear information inequality},
respectively \emph{linear rank inequality},
on $m$ variables
consists of a collection $(\alpha_I)_{I \in \Par([m])}$
of real numbers such that
\(
\sum_{I \in \Par([m])} \alpha_I f(A_I) \ge 0
\)
for every entropic,
respectively linear, polymatroid $(Q,f)$
and for every collection $(A_i)_{i \in [m]}$ of subsets of $Q$.
%A \emph{linear information inequality} 
%(respectively, \emph{linear rank inequality})
%consists of a sequence $(\alpha(X))_{X \in \Par(Q)}$
%of real numbers such that
%\(
%\sum_{X \in \Par(Q)} \alpha(X) f(X) \ge 0
%\)
%for every entropic (respectively, linearly representable) polymatroid $(Q,f)$.
Since every linear polymatroid is entropic,
every information inequality is also a rank inequality.

\emph{Shannon information inequalities} are those that are 
derived from the polymatroid axioms in Definition~\ref{df:polym1}. 
Ingleton inequality~\cite{Ing71},
which can be written in a compact form as 
\begin{equation}
\label{eq:Ing}
f(A_2 \lon A_3) \le f(A_2 \lon A_3|A_1) + f(A_2 \lon A_3|A_4) + f(A_1 \lon A_4)
\end{equation}
was the first known example of a non-Shannon linear rank inequality.
The information inequality
\begin{equation}
\label{eq:ZY}
2 f(A_2 \lon A_3) \le f(A_1 \lon A_4) + f(A_1 \lon A_2 A_3) + 
3 f(A_2 \lon A_3| A_1) + f(A_2 \lon A_3| A_4)
\end{equation}
which was presented by Zhang and Yeung~\cite{ZhYe98},
was the first known example of a non-Shannon linear information inequality.

Folded linear matroids are entropic.
Every linear matroid is algebraic~\cite{Oxl92}.
It has been recently proved that
every algebraic matroid is almost entropic~\cite{Matus17}.
Vamos matroid is not almost entropic because 
it does not satisfy Zhang--Yeung inequality.
Non-Pappus matroid is a folded linear matroid
that is algebraic but not linear~\cite{Oxl92,SiAs98}.
Two examples of almost entropic matroids
that are not entropic were given in~\cite[Remarks 4, 5]{Matus18}.
Only one of them is algebraic.
A folded linear matroid that
is not algebraic was presented in~\cite{Ben16}.
It is not known if there exist entropic matroids that are not folded linear.
These facts are illustrated in Figure~\ref{classes}.

For every positive integer $k$
and any field $\field$,
the class of $k$-folded $\field$-linear 
matroids is closed by duality~\cite{JaMa94,Oxl92}.
It is unknown whether or not this is the case for 
the classes of algebraic or entropic matroids.
Remarkably, Kaced~\cite{Kac18} recently proved
that the class of almost entropic matroids is not closed by duality.
An explicit counterexample is presented in~\cite{Csi19}.

Every minor of an $\field$-linear polymatroid is $\field$-linear. 
That is, the class of $\field$-linear polymatroids is closed under minors.
The same applies to the class of almost entropic polymatroids~\cite[Lemma~1]{MaCs13}. 
The classes of linear, folded linear, algebraic~\cite[Corollary~6.7.14]{Oxl92},
and almost entropic matroids are closed under minors. 

%%%%%%%%%%%%%%%%%%%%%%%%%%%%%%%%%%%%%%
\subsection{Secret Sharing}
\label{pt:presss}
%%%%%%%%%%%%%%%%%%%%%%%%%%%%%%%%%%%%%%%

\begin{definition}
An \emph{access function} on a finite set $P$ is a map
$\Gamma \colon \Par(P) \to \Real$ satisfying the following properties.
\begin{enumerate}
\item
$\Gamma(\emptyset) = 0$ and $\Gamma(P) = 1$.
\item
$\Gamma(X) \le \Gamma(Y)$ if $X \subseteq Y \subseteq P$.
\end{enumerate}
An access function is \emph{perfect} if 
its only values are $0$ and $1$.
The \emph{qualified} and \emph{forbidden} 
sets of the access function $\Gamma$ are
the ones with $\Gamma(X) = 1$ and, respectively, $\Gamma(X) = 0$. 
\end{definition}

\begin{definition}
For a polymatroid $(Q,f)$ and a point $p_o \in Q$
with $f(p_o) > 0$ and $f(Q \sss p_o) = f(Q)$,
the \emph{port of the polymatroid $(Q,f)$ at $p_o$}
is the access function $\Gamma$ on the set $P = Q \sss p_o$
defined by 
\[
\Gamma(X) = \frac{f(X \lon p_o)}{f(p_o)}.
\]
\end{definition}

The \emph{dual} $\Gamma^*$ 
of an access function $\Gamma$ on $P$
is defined by $\Gamma(X) = 1 - \Gamma(P \sss X)$ for every $X \subseteq P$.
If $\Gamma$ is the port of a matroid $M$ at $p_o$,
then its dual $\Gamma^*$ is the port of the dual matroid
$M^*$ at $p_o$.
Consider an access function $\Gamma$ on $P$ and
a subset $B \subseteq P$.
If $\Gamma(P \sss B) = 1$, 
the access function $\Gamma \setminus B$
on $P \sss B$ defined by 
$(\Gamma \setminus B)(X) = \Gamma(X)$
is the \emph{deletion of $B$ from $\Gamma$}.
If $\Gamma(B) = 0$, the access function $(\Gamma/B)$ with 
$(\Gamma/B)(X) = \Gamma(XB)$ is the 
\emph{contraction of $B$ from $\Gamma$}.
Every access function that is obtained from 
$\Gamma$ by deletions and contractions is
a \emph{minor} of $\Gamma$.
If $\Gamma$ is the port of a polymatroid $M = (Q,f)$ at $p_o$
and $B \subseteq P = Q \sss p_o$, 
then the minors $\Gamma \setminus B$ and $\Gamma/B$ are
the ports of $M \setminus B$ and, respectively, $M/B$ at $p_o$.

\begin{definition}
Let $P$ be a finite set of \emph{players}
and $Q = P p_o$ with $p_o \notin P$.
Let $\Gamma$ be an access function on $P$.
Let $S = (S_x)_{x \in Q}$ be a discrete random vector
and $(Q,h)$ the entropic polymatroid determined by $S$.  
Then $S$ is a \emph{secret sharing scheme} on $P$
with access function $\Gamma$ if the 
following properties are satisfied.
\begin{enumerate}
\item
$h(p_o) > 0$ and $h(P) = h(P p_o)$.
\item
$\Gamma$ is the port of $(Q,h)$ at $p_o$. 
\end{enumerate} 
The random variable $S_{p_o}$ corresponds to
the \emph{secret value}, and the \emph{share}
for a player $x \in P$ is given by the random variable $S_x$.
\emph{Linear} secret sharing schemes are those defined by
linear random vectors. 
A secret sharing scheme is \emph{perfect}
if its access function is perfect.
The \emph{information ratio} of a secret sharing scheme
is $\max_{x \in P} h(x)/h(p_o)$, that is,
the ratio between the maximum length of the shares
and the length of the secret.
\end{definition}

Only perfect secret sharing schemes are going to be considered in this work.
Perfect access functions are also called \emph{access structures}.
Each of them is determined by its minimal qualified sets.
An access structure is \emph{connected} if every player
is in some minimal qualified set.
All access structures in this paper
are supposed to be connected.
In a perfect scheme, 
$h(x) \ge h(p_o)$ for every $x \in P$. 
A perfect secret sharing scheme is \emph{ideal}
if $h(x) = h(p_o)$ for every $x \in P$.
The \emph{optimal information ratio} $\sigma(\Gamma)$
of an access structure $\Gamma$ is the infimum
of the information ratios of the secret sharing schemes for $\Gamma$,
while $\lambda(\Gamma)$ is the corresponding value when
restricting the optimization to linear secret sharing schemes.

A matroid is \emph{connected} if every pair of points
in the ground set lie in a common circuit. 
All ports of a connected matroid are connected access structures.
Moreover, a connected matroid is determined by any of its ports.

Let $S = (S_x)_{x \in Q}$ be an ideal secret sharing scheme
and let $h$ be the entropic vector associated to $S$.
Then the polymatroid $(Q,f)$ defined by
$f(X) = h(X)/h(p_o)$ for every $X \subseteq Q$ is a matroid~\cite{BrDa91}.
As a consequence, the access structures
of ideal secret sharing schemes coincide with the
ports of entropic matroids, and the ports
of folded linear matroids
are precisely the access structures of ideal linear secret sharing schemes.
%The values of $\sigma$ and $\lambda$ for the ports of a 
%matroid $M$ can be seen as a measure of how far is $M$ 
%from being entropic and multilinearly representable, respectively.

%%%%%%%%%%%%%%%%%%%%%%%%%
%%%%%%%%%%%%%%%%%%%%%%%%%
\section{How to Use Undiscovered Information and Rank Inequalities}
\label{sec:InfIneq}
%%%%%%%%%%%%%%%%%%%%%%%%%
%%%%%%%%%%%%%%%%%%%%%%%%%

The title of this section is borrowed from~\cite{GuRo2019}.
It precisely describes the main idea behind the technique introduced in~\cite{FKMP20},
namely, using properties from which information and rank inequalities have been derived
instead of using known inequalities.

%%%%%%%%%%%%%%%%%%%%%%%%%%%%%%%%%%%%%%%%
\subsection{Common Information}\label{subs:CI}
\label{pt:cip}
%%%%%%%%%%%%%%%%%%%%%%%%%%%%%%%%%%%%%%%%

We say that a random variable $S_3$
\emph{conveys the common information\/}
of the random variables $S_1$ and $S_2$
if $H(S_3 | S_2) = H(S_3 | S_1) = 0$
and $H(S_3) = I(S_1 \lon S_2)$.
In general, given two random variables,
it is not possible to find 
a third one satisfying those conditions~\cite{GaKo73}.
Nevertheless, this is possible for every
pair of random variables in a linear random vector.
%Oriol: all -> many
Most of the known non-Shannon rank inequalities
are derived from this fact~\cite{DFZ09}. 
A combinatorial abstraction of the notion of common information is
given in the next definition.

\begin{definition}
Let $(Q,f)$ be a polymatroid 
and let $A, B \subseteq Q$.
Then every subset $X_o \subseteq Q$ satisfying
\begin{description}
\item[(C1)]
$f(X_o | A) =  f(X_o | B) = 0$, and
\item[(C2)]
$f(X_o) = f(A \lon B)$
\end{description}
is called a \emph{common information for the pair $(A,B)$}.
If $X_o = \{x_o\}$, then the element $x_o$ is also called a
common information for the pair $(A,B)$.
\end{definition}

\begin{definition}
Consider polymatroids $(Q,f)$
and $(Q',f')$ with $Q \subseteq Q'$.
We say that $(Q',f')$ is an \emph{extension}
of $(Q,f)$ if $f(X) = f'(X)$ for every $X \subseteq Q$.
In this situation we will generally use the same symbol
for both rank functions.   
\end{definition}

\begin{definition}
\label{def:kcip}
A polymatroid $(Q,f)$ is
\emph{1-CI-compliant} if,
for every pair $(A,B)$
of subsets of $Q$, there exists an extension
$(Q x_o,f)$ such that $x_o$ is a common information
for the pair $(A,B)$.  
Inductively, for every integer $k > 1$,
a polymatroid $\Poly=(Q,f)$ is
\emph{$k$-CI-compliant} if,
for every pair $(A,B)$
of subsets of $Q$, there exists an extension
$(Q x_o,f)$ such that $x_o$ is a common information
for the pair $(A,B)$ and $(Q x_o,f)$ is $(k-1)$-CI-compliant.
A polymatroid is \emph{CI-complaint} if it is 
$k$-CI-compliant for every positive integer $k$.
\end{definition}

\begin{proposition}
Let $\field$ be a field.
Consider an $\field$-linear polymatroid
$(Q,f)$ and a pair $(A,B)$ of subsets of the ground set.
Then there exists an $\field$-linear extension
$(Qx_o,f)$ such that $x_o$ is a common information for $(A,B)$.
As a consequence, linear polymatroids and, in particular,
folded linear matroids are CI-compliant. 
\end{proposition}

\begin{proof}
Consider a collection $(V_x)_{x\in Q}$ of vector subspaces
providing an $\field$-linear representation of $(Q,f)$.
For every $X \subseteq Q$, put $V_X = \sum_{x \in X} V_x$.
Given a pair $(A,B)$ of subsets of $Q$, 
take $V_{x_o} = V_A \cap V_B$. 
Then $(V_x)_{x \in Qx_0}$ is an $\field$-linear representation
of a polymatroid $(Qx_0,f)$ extending $(Q,f)$
in which $x_0$ is a common information for $(A,B)$. 
\end{proof}

%Oriol: paragraph added.
%In this work we do not consider rank inequalities 
%that are characteristic-dependent, that is, rank inequalities that hold 
%for fields of a specific characteristic. 
%%%%%%%%%%%%%%%%%%%%%%%%%%%%%%%%%%%%
\subsection{Ahlswede and K\"{o}rner's Information}\label{subs:AK}
%%%%%%%%%%%%%%%%%%%%%%%%%%%%%%%%%%%%

Linear information inequalities can be derived 
from properties that are satisfied 
by every almost entropic polymatroid.
Specifically, all known 
linear information inequalities
have been derived from the copy lemma~\cite{ZhYe98} 
and the Ahlswede--K\"orner lemma~\cite{AhKo77,AhKo06,CsKo81} 
as used in~\cite{MMRV02}.

\begin{definition}
Let $(Q,f)$ be a polymatroid, and let $U,V,Z \subseteq Q$.
Then every subset $Z_o \subseteq Q$ such that
	\begin{description}
	\item[(AK1)] $f(Z_o|UV) = 0$,  
	\item[(AK2)] $f(U|Z_o) = f(U|Z)$ and $f(V|Z_o) = f(V|Z)$,
	\item[(AK3)] $f(UV|Z_o) = f(UV|Z)$
	\end{description}
	is called an \emph{AK-information for the triple $(U,V,Z)$}.
\end{definition}

We say that a polymatroid $(Q,f)$ is \emph{1-AK-compliant} if, 
for every triple $(U,V,Z)$ of subsets of $Q$, 
there exists an extension $(Q z_o, f)$ 
such that $z_o$ is an AK-information for the triple $(U,V,Z)$. 
Analogously to the discussion on the common information property, 
we can define \emph{$k$-AK-compliance} for every $k > 0$ and also \emph{AK-compliance}.
Next proposition was proved in~\cite{FKMP20}
from~\cite[Lemma~5]{MMRV02} and~\cite[Lemma~2]{Kac13}.
As a consequence, almost entropic polymatroids are AK-compliant.

\begin{proposition}
\label{lemma:AK}
For every almost entropic polymatroid
$(Q,f)$ and sets $U,V,Z \subseteq Q$,
there exists an almost entropic extension $(Q z_o,f)$ such that
$z_o$ is an AK-information for the triple $(U,V,Z)$. 
\end{proposition}

As consequence of the following result from~\cite{FKMP20},
$k$-CI-compliant polymatroids are also $k$-AK-compliant. 

\begin{proposition}
\label{st:CI2AK}
If $x_o$ is a common information for the pair $(UV,Z)$, 
then $x_o$ is an AK-information for the triple $(U,V,Z)$. 
\end{proposition}

%%%%%%%%%%%%%%%%%%%%%%%%%%%%%%%%%%%%%%%%%%%%%%%
\subsection{Application to Secret Sharing}
\label{sec:SecShaSch}
%%%%%%%%%%%%%%%%%%%%%%%%%%%%%%%%%%%%%%%%%%%%%%%%

We describe next the linear programming technique 
that has been extensively used (see the references in~\cite{FKMP20})
to find lower bounds in secret sharing
and the improvement on it proposed in~\cite{FKMP20}.
 
Let $(S_x)_{x \in Q}$ be a secret sharing scheme 
with access structure $\Gamma$ on the set of players
$P = Q \sss p_o$.
Let $(Q,h)$ be the entropic polymatroid
determined by it and take the polymatroid
$(Q,f)$ given by $f(X) = h(X)/h(p_o)$.
Then the vector $(f(X))_{X \in \Par(Q)}$
satisfies the linear constraints
\begin{description}
\item[(N)] $f(p_o) = 1$,
\item[($\Acc$)]
$f(X \lon p_o) = \Gamma(X)$ for every $X \subseteq P$
\end{description}
and also the polymatroid axioms (P1)--(P3) in Definition~\ref{df:polym1}.
Therefore, the vector $f$ is a feasible solution
of Linear Programming Problem~\ref{tab:LPkap1}.

\begin{lpp}
\label{tab:LPkap1}
For an access structure $\Gamma$ on the set $P$, 
the optimal value of this linear programming problem
is, by definition, $\kappa(\Acc)$.
	\begin{align*}
	\text{Minimize }\quad& v \\
	\text{subject to}\quad& v \ge f(x)  \text{ for every } x \in P\\
	&   \mathrm{(N)}, (\Acc), \mathrm{(P1), (P2), (P3)}
	\end{align*}
\end{lpp}

Since this applies to every secret sharing scheme 
with access structure $\Acc$ and the objective function
equals the information ratio, the optimal value $\kappa(\Acc)$ of this
linear programming problem is a lower bound on $\sigma(\Acc)$.
It is the best lower bound that can be obtained by using only
Shannon information inequalities~\cite{Csi97,MaPa10}.
That linear program can be improved
by adding non-Shannon information inequalities~\cite{BLP08,Met11,PVY13}
or, as proposed in~\cite{FKMP20},
constraints derived from AK-information
or common information.

\begin{lpp}
\label{tab:LPavlamci}
Consider an access structure $\Gamma$ on a set $P$
and a pair $(A_0,A_1)$ of subsets of $P$. 
The optimal value of this linear programming problem 
is a lower bound on $\lambda(\Acc)$.
	\begin{align*}
	\text{Minimize }\quad& v \\
	\text{subject to}\quad& v \ge f(x)  \text{ for every } x \in P\\
	&\mathrm{(N)}, (\Acc)\\
	%&f(x_o | A_0) = 
	%f(x_o | A_1) = 0\\
	%&  f(x_o) = f(A_0) + f(A_1) - f(A_0 \, A_1)\\
	& \mathrm{(C1), (C2)}\text{ for } (A_0,A_1) \text{ and } x_o\\
	&\mathrm{(P1), (P2), (P3)} \text{ on the set } Q x_o. \\
	\end{align*}
\end{lpp}

\begin{lpp}
	\label{tab:LPavlamak+}
	Let $U,V,Z\subseteq P$. The optimal value of this linear programming problem 
	is a lower bound on $\sigma(\Acc)$.
	\begin{align*}
	\text{Minimize }\quad& v \\
	\text{subject to}\quad& v \ge f(x)  \text{ for every } x \in P\\
	&\mathrm{(N)}, (\Acc)\\
	& \mathrm{(AK1), (AK2), (AK3)}\text{ on }z_0\text{ and } (U,V,Z)\\
	& \mathrm{(P1), (P2), (P3)} \text{ on the set } Q z_o. 
	\end{align*}
\end{lpp}

These linear programming problems can be extended by adding
the common information or the AK-information
for more pairs or, respectively, triples of sets.

%%%%%%%%%%%%%%%%%%%%%%%%%%%%%%%%%%%%%%%%%%%%%%%%%%%%%%%%%%%%%%%
\subsection{Application to Classification of Matroids}
%%%%%%%%%%%%%%%%%%%%%%%%%%%%%%%%%%%%%%%%%%%%%%%%%%%%%%%%%%

Linear information inequalities provide necessary
conditions for a matroid to be almost entropic
and, as a consequence of the result in~\cite{Matus17},
also to be algebraic.
The same applies to linear rank inequalities 
with respect to the class of folded linear matroids.
A polymatroid is \emph{Ingleton-compliant},
respectively \emph{ZY-compliant},
if Ingleton inequality~(\ref{eq:Ing}),
respectivey Zhang--Yeung inequality~(\ref{eq:ZY}),
holds for every  collection
$(A_i)_{i \in [4]}$ of subsets of the ground set. 
As a consequence of the proofs for 
those inequalities~\cite{DFZ09,Kac13,MMRV02},
$1$-CI-compliant and $1$-AK compliant polymatroids
are Ingleton-compliant and, respectively, ZY-compliant.
Those inequalities are related to a special
configuration introduced in~\cite{AlHo95}.

\begin{definition}
A matroid $(Q,r)$ satisfies the \emph{bundle condition} 
if it does not contain four flats $(A_i)_{i \in [4]}$ such that 
every flat has rank $2$, the union of every pair of flats
has rank $3$ except for $r(A_1 A_4) = 4$, 
and the union of every three or four flats has rank $4$.
\end{definition}

Vamos matroid is among the smallest ones
violating the bundle condition,
and the one with the minimum
number of dependent hyperplanes. 
If a matroid does not satisfy the bundle condition,
then the collection $(A_i)_{i \in [4]}$ described in
the previous definition violates both
Ingleton and Zhang--Yeung inequalities
as expressed in~(\ref{eq:Ing}) and~(\ref{eq:ZY}), respectively.
Therefore, almost entropic matroids and, in particular,
algebraic matroids satisfy the bundle condition. 
Moreover, the sparse paving matroids 
that are Ingleton-compliant coincide with those
satisfying a generalization of the bundle 
condition~\cite[Corollary~3.2]{NP18}.

\begin{proposition}
\label{t:SPIC}
Let $M$ be a sparse paving matroid of rank $k \ge 4$. 
Then $M$ is not Ingleton-compliant if and only if 
there exist five pairwise disjoint subsets
$B$, $A_1, A_2, A_3,$ $A_4$ of the ground set 
with $|B| = k - 4$ and $|A_i| = 2$ 
such that  $B A_1 A_4$ is a basis
and all the other sets of the form $B A_i A_j$ with $i\neq j$
are circuit-hyperplanes.
\end{proposition}

\begin{corollary}
\label{st:minIng}
If a sparse paving matroid $M$ is not Ingleton-compliant, 
then there is a minor of $M$ on eight points that is not Ingleton-compliant. 
\end{corollary}

As a consequence, the class of Ingleton-compliant
sparse paving matroids has a finite number of 
forbidden minors~\cite[Theorem 1.3]{NP18}.
In contrast, the set of excluded minors
for the class of Ingleton-compliant matroids is infinite~\cite{MNW09}.
By combining Proposition~\ref{t:SPIC} with 
a recent result about algebraic matroids~\cite{Matus17},
the following remarkable property of sparse paving matroids is
easily derived.

\begin{theorem}
\label{st:Ingspar}
If a sparse paving matroid is not Ingleton-compliant, 
then it is not ZY-compliant
and hence it is neither almost entropic nor algebraic.
\end{theorem}

\begin{proof}
If a sparse paving matroid admits the configuration described in
Proposition~\ref{t:SPIC},
then Zhang--Yeung inequality~(\ref{eq:ZY})
does not hold for $(B A_i)_{i \in [4]}$.
\end{proof}

By using the result in Proposition~\ref{t:SPIC}, 
Nelson and van der Pol~\cite{NP18}
proved that the number of
Ingleton-compliant matroids is
doubly exponential on the size of the ground set.
This indicates that the power of Ingleton inequality
in the classification of matroids is quite limited.
Of course, many more rank and information inequalities 
are available, but one may expect a better outcome from
the strategy introduced in~\cite{FKMP20},
which makes it possible to use 
undiscovered inequalities.
This claim is supported by the results 
obtained in secret sharing~\cite{FKMP20,GuRo2019}.
Specifically, the linear programming technique discussed
in Section~\ref{sec:SecShaSch}
can be adapted to the study of the classes of matroids
described in Section~\ref{pt:matbas}
by using the following linear programming problems
or their extensions to multiple pairs or triples of sets.

\begin{lpp}\label{ci4rep}
Given a polymatroid $(Q,r)$, and $A,B \subseteq Q$, 
determine if there is an extension $(Qx_o,r)$
such that $x_o$ is a common information for the pair $(A,B)$.  
\end{lpp}

\begin{lpp}\label{ak4almentr}
Given a polymatroid $(Q,r)$ and $U,V,Z \subseteq Q$, 
determine if there is an extension $(Q z_o,r)$
such that $z_o$ is an AK-information for the triple $(U,V,Z)$.
\end{lpp}

Those linear programming problems 
can be used to disprove that a given matroid
is folded linear or almost entropic. 
To that end, one can also apply
Linear Programming Problems~\ref{tab:LPavlamci}
or~\ref{tab:LPavlamak+} (or their extensions)
to any port of the given matroid.
The corresponding common information or AK-information
exists if and only if the optimal value is equal to $1$.

Nevertheless, that technique is useless for matroids of rank $3$
because they are CI-compliant.
This is easily proved by using the results about 
\emph{single-element extensions}
and \emph{modular cuts} of 
matroids in~\cite[Section 7.2]{Oxl92}.

\begin{proposition}
\label{st:r3ci}
Every matroid of rank $3$ is CI-compliant,
and hence also AK-compliant.
\end{proposition}

\begin{proof}
It is enough to  prove that,
for every matroid $(Q,r)$ of rank $3$
and for every pair $(A,B)$ of subsets of $Q$,
there exists a matroid  $(Q x_o,r)$ of rank $3$
extending $(Q,r)$ such that 
$x_o$ is a common information for the pair $(A,B)$.
Obviously, it is enough to prove the result for 
pairs of hyperplanes.
Let $M = (Q,r)$ be a matroid of rank $3$
and let $(H_1,H_2)$ be a pair of 
distinct hyperplanes of $M$.
If there exists $x_o \in H_1 \cap H_2$
with $r(\{x_o\}) = 1$,
then $x_o$ is a common information for the  
pair $(H_1,H_2)$.
Otherwise $(H_1,H_2)$ is not a modular pair
and hence the set of flats $\{H_1,H_2,Q\}$
is a modular cut of the matroid $(Q,r)$.
Therefore, by~\cite[Theorem 7.2.3]{Oxl92}, $M$ can be extended to 
a matroid $(Q x_o, r)$ of rank $3$ such that
$r(H x_o) = r(H)$ if $H \in \{H_1,H_2,Q\}$
and $r(H x_o) = r(H) + 1$ if $H$ is any other flat of $M$. 
Clearly, $x_o$ is a common information for the pair $(H_1,H_2)$.  
\end{proof}

\section{Classification of Matroids on 8 Points}\label{subs:8-ptMat}

The matroids $AG(3,2)$, $AG(3,2)'$, $F_8$, $Q_8$, 
$V_8$ (Vamos matroid), $P_8$, and $L_8$ 
appearing in this section
and in Section~\ref{sec:SSSforMatrPort} are described 
in the Appendix of Oxley's book~\cite{Oxl92}.
Given a sparse paving matroid $M$, 
a new such matroid $M'$
can be obtained by \emph{relaxing} one of its
circuit-hyperplanes, that is, by transforming it into a basis.
In that situation, $M'$ is called a 
\emph{relaxation} of $M$.

\subsection{Matroids that are not Ingleton-compliant}
\label{pt:the39}

Mayhew and Royle~\cite{MaRo08} provided
a comprehensive list of matroids on up to $9$ points, 
specifying how many of them are simple, paving, or sparse paving. 
They also presented the list of all $44$
non-linear matroids on $8$ points,
which are sparse paving and of rank $4$.
Since every matroid on at most $7$ points is linear,
those are the smallest non-linear matroids.
Exactly $39$ of them are not Ingleton-compliant, 
which implies by Theorem~\ref{st:Ingspar} 
that they are neither almost entropic nor algebraic.
Those $39$ matroids, which include $F_8$ and $Q_8$, 
are relaxations of the binary affine cube $AG(3,2)$, 
with $AG(3,2)'$ and the Vamos matroid 
$V_8$ the ones among them with, respectively,
most and fewest circuit-hyperplanes. The matroids in~\cite{MaRo08} are named according to the database provided by the same authors in~\cite{RoMaDatabase}. In this work we follow the same notation.

\subsection{Folded Linear Matroids}
\label{pt:folded}

The 5 remaining non-linear matroids on $8$ points
are $P_1$, $P'_2$, $P''_2$, and $P_3$,
which are relaxations of $P_8$, 
and a relaxation $L'_8$ of $L_8$.
Take $Q = \{0,1,\ldots,7\}$
as the ground set of those sparse paving matroids.
The circuit-hyperplanes of $P_8$ are
\[
0127, 0136, 0235, 1234, 
0456, 1457, 2467, 3567,
0347, 1256,
\]
while the ones of $L_8$ are
\[
0246,1357, 0156, 2347, 0127, 3456, 0457, 1236.
\]
The matroid $P_1$ is obtained from $P_8$ 
by relaxing the circuit-hyperplane $3567$ of $P_8$.
The relaxation of $0347$ from $P_1$ gives the matroid $P'_2$,
while $P''_2$ is obtained from $P_1$ by relaxing $1256$.
The relaxation of both $0347$ and $1256$
from $P_1$ produces the matroid $P_3$. 
Finally, the matroid $L'_8$ is obtained from $L_8$ by 
relaxing the circuit-hyperplane $0457$.

By applying Linear Programming Problem~\ref{ci4rep} 
to those five non-linear matroids, 
we found out that they are 1-CI-compliant,
and hence also 1-AK-compliant by Proposition~\ref{st:CI2AK}.
We explored the possibility that some
of them were folded linear matroids.
To that end, we combined the technique 
to find linear representations of matroids
presented in~\cite[Section 6.4]{Oxl92}
with the tools for  folded linear matroids
given in~\cite{BBPT14} and we concluded 
that only $P_3$ and $L'_8$ are folded linear matroids.

%In this work, we used the Gurobi\textsuperscript{TM} optimizer for solving the linear programming problems, and the Sage\textsuperscript{\textregistered} matroid package for specific matroid operations.

\begin{theorem}
\label{thm:multilinrep}
The smallest non-linear matroids that are  folded linear are precisely $P_3$ and $L'_8$.
\end{theorem}

Before proving Theorem~\ref{thm:multilinrep}, 
we describe how to use the techniques
from~\cite{BBPT14,Oxl92} to that end.
Unless otherwise stated, 
the blocks in the matrices appearing in this section are 
square matrices of size $\ell$.
We use capital letters to represent them.
As usual, the identity and zero matrices
are denoted by $I$ and $0$, respectively.

Consider a  matroid $M = (Q,r)$ of rank $m$ on $n$ points,
a field $\field$, and a positive integer~$\ell$.
Assume that $Q = \{0,1,\ldots, n - 1\}$
is the ground set of $M$.
Every $\field$-linear representation
of the polymatroid $(Q, \ell r)$ is called 
an \emph{$(\field,\ell)$-linear representation of $M$},
and it is determined by a block matrix over $\field$
of the form
\begin{equation}
\label{eq:genblock}
B=\left(
\begin{array}{ccc}
	{B_{0,0}} & \cdots & {B_{0,n-1}} \\
	\vdots &  & \vdots  \\
 {B_{m-1,0}} & \cdots &  {B_{m-1,n-1}}  \\        
\end{array} \right),
\end{equation}
where each block $B_{i,j}$ is a 
square matrix of size $\ell$.
If $V_i$ is the vector subspace of $\field^{\ell m}$
spanned by the columns in the $i$-th block-column,
then $(V_i)_{i \in Q}$ is an $\field$-linear 
representation of the polymatroid $(Q,\ell r)$.
By the next result, there exists such a matrix
in which every block is either invertible or zero.

\begin{lemma}
\label{lemma:invblks}
Suppose that $A=\{0,1,\ldots,m-1\}$ is a basis of $M$.
For each $j = m, \ldots, n-1$, consider the
fundamental circuit $C(j,A)$, that is, the only 
circuit contained in $A \cup j$.
Then there exists a block matrix of the form
\begin{equation}\label{eqn:basismatrix}
\left(
\arraycolsep=4pt
\begin{array}{ccc|ccc}
{I} & \cdots & {0} & {B_{0,m}} & \cdots & {B_{0,n-1}} \\
\vdots & \ddots & \vdots & \vdots &  & \vdots \\
{0} & \cdots & {I} & {B_{m-1,m}} & \cdots & {B_{m-1,n-1}}  \\        
\end{array} \right),
\end{equation}
providing an $(\field,\ell)$-linear representation of $M$.
Furthermore, in every such representation,
each block $B_{i,j}$ with $j \ge m$ is 
invertible if $i \in C(j,A)$ and it is zero otherwise.
\end{lemma}

\begin{proof}
If $B'$, a block matrix of the form~\eqref{eq:genblock},
is an $(\field,\ell)$-linear representation of $M$, 
then the submatrix $T$ formed by
the block-columns corresponding to the basis $A$ is invertible.
Clearly, $B = T^{-1} B'$ is 
an $(\field,\ell)$-linear representation of $M$
of the form~(\ref{eqn:basismatrix}).
Consider $j \ge m$.
Without loss of generality, suppose
that $C(j,A) = \{0, \ldots, s-1,j\}$ for some $s \le m$.
Since the submatrix of $B$ formed by the 
block-columns corresponding to $C(j,A)$ has rank $\ell s$,
it is clear that $B_{i,j} = 0$ if $s \le i \le m-1$.
If, otherwise, $0 \le i \le s-1$, the rank of the submatrix
formed by the block-columns corresponding to
$C(j,A)\sss i$ equals $\ell s$, 
which implies that $B_{i,j}$ is invertible.
\end{proof}

Following~\cite{BBPT14}, we are going to use two operations
on block matrices representing folded linear matroids.
Namely, \emph{block-column scaling} 
and \emph{row-block scaling}.

\begin{lemma}[\cite{BBPT14} Proposition 2.12]\label{lemma:invblks2}
Let $M$ be an $\ell$-folded linear matroid
represented by a block matrix $B$ of the form~\eqref{eq:genblock}
and let ${G}$ be an invertible $\ell\times \ell$ matrix.
Then, for each $i = 0, \ldots, m-1$,  the matrix 
\[\left(
\begin{array}{ccc}
{B_{0,0}} & \cdots & {B_{0,n-1}} \\
\vdots &  & \vdots  \\
{G} {B_{i,0}} & \cdots & {G} {B_{i,n-1}} \\
\vdots &  & \vdots  \\
{B_{m-1,0}} & \cdots &  {B_{m-1,n-1}}  \\        
\end{array} \right)
\]
is also an $(\field,\ell)$-linear representation of $M$,
and the same applies to the matrix
\[
\left(
\begin{array}{ccccc}
{B_{0,0}} & \cdots & {B_{0,j}} {G} & \cdots & {B_{0,n-1}} \\
\vdots &  & \vdots  &  & \vdots \\
{B_{m-1,0}} & \cdots & {B_{m-1,j}} {G} &  \cdots & {B_{m-1,n-1}}  \\        
\end{array} \right)
\]
for each $j = 0, \ldots, n-1$.
\end{lemma}
%%%%%%% Proposition B.2 a & b from [7] as a Lemma %%%%% 

Block scaling can help significantly in simplifying the study of 
$(\field,\ell)$-linear representations. 
By the following lemma, we can assume that several blocks ${B_{i,j}}$ 
in~\eqref{eqn:basismatrix} equal the identity matrix.
It is a straightforward generalization of~\cite[Theorem~6.4.7]{Oxl92},
the analogous result for linear representations of matroids.

\begin{lemma}
\label{ob:representation}
Let $\Mat$ be an $\ell$-folded $\field$-linear matroid
that admits an $(\field,\ell)$-representation $B'$
of the form~\eqref{eqn:basismatrix}.
Take $V = \{0, \ldots, m-1\}$
and $W = \{m, \ldots, n-1\}$.
Consider the bipartite graph $G$
with set of vertices $V \cup W$
such that $(i,j) \in V \times W$ 
is an edge if and only if $B'_{i,j} \ne 0$.
Let $E$ be the set of edges of a maximal
acyclic subgraph of $G$.
Then a sequence of block scalings
provides an  $(\field,\ell)$-representation $B$
of the form~\eqref{eqn:basismatrix}
such that $B_{i,j} = I$ if $(i,j) \in E$.
\end{lemma}

\begin{proof}
Adapt the proof of~\cite[Theorem~6.4.7]{Oxl92}
in the obvious way.
\end{proof}

The graph $G$ is connected for many matroids, and
in that case we can choose
any spanning tree of $G$ and we can assume that
the $n-1$ blocks $B_{i,j}$ with $j \ge m$
corresponding to its edges are equal to $I$.
We are now ready to prove Theorem~\ref{thm:multilinrep}.

\begin{proof}[Proof of Theorem~\ref{thm:multilinrep}]
Let $M$ be one of the matroids
$P_1, P_2',P_2'',P_3$
and suppose that it is an
$\ell$-folded $\field$-linear matroid
for some field $\field$ and some positive integer $\ell$.
Since $0123$ is a basis, 
by Lemmas~\ref{lemma:invblks} and~\ref{ob:representation},
we can assume that $M$ admits an 
$(\field,\ell)$-linear representation of the form 
\begin{equation}
\label{pmats:rep}
\left(
\arraycolsep=4pt
\begin{array}{cccc|cccc}
{I} & {0} & {0} & {0} & {0} & {I} & {I} & {I} \\
{0} & {I} & {0} & {0} & {I} & {0} & {I} & {A} \\
{0} & {0} & {I} & {0} & {I} & {B} & {0} & {C} \\        
{0} & {0} & {0} & {I} & {I} & {D} & {E} & {0} \\        
\end{array}
\right)
\end{equation}
where all nonzero blocks are invertible.
We next consider the circuit-hyperplanes 
$0456$, $1457$, and $2467$.
The submatrices corresponding to those sets are, respectively,
\[\left(
\arraycolsep=2pt
\begin{array}{cccc}
{I} & {0} & {I} & {I} \\
{0} & {I} & {0} & {I} \\
{0} & {I} & {B} & {0} \\
{0} & {I} & {D} & {E} \\
\end{array}\right), 
\left(
\arraycolsep=2pt
\begin{array}{cccc}
{0} & {0} & {I} & {I} \\
{I} & {I} & {0} & {A} \\
{0} & {I} & {B} & {C} \\
{0} & {I} & {D} & {0} \\
\end{array}\right), \text{ and }  
\left(
\arraycolsep=2pt
\begin{array}{cccc}
{0} & {0} & {I} & {I} \\
{0} & {I} & {I} & {A} \\
{I} & {I} & {0} & {C} \\
{0} & {I} & {E} & {0} \\
\end{array}\right).
\]
Each of these matrices has rank $3 \ell$.
Gaussian elimination transforms those matrices into
\[\left(
\arraycolsep=2pt
\begin{array}{cccc}
{I} & {0} & {I} & {I} \\
{0} & {I} & {0} & {I} \\
{0} & {0} & {B} & {-I} \\
{0} & {0} & {0} & {DB^{-1}+E-I} \\
\end{array}\right), 
\left(
\arraycolsep=2pt
\begin{array}{cccc}
{I} & {I} & {0} & {A} \\
{0} & {I} & {D} & {0} \\
{0} & {0} & {I} & {I} \\
{0} & {0} & {0} & {C-B+D} \\
\end{array}\right), \text{ and }
\left(
\arraycolsep=2pt
\begin{array}{cccc}
{I} & {I} & {0} & {C} \\
{0} & {I} & {E} & {0} \\
{0} & {0} & {I} & {I} \\
{0} & {0} & {0} & {A-I+E} \\
\end{array}\right).
\]
Therefore,
\begin{align}
\label{pmats:0456} {D}&=({I-E}){B}\\
\label{pmats:1457} {C}&={B-D}={E}{B}\\
\label{pmats:2467}  {A}&={I-E}
\end{align}
Since $3567$ is a basis, 
the corresponding submatrix has full rank.
Gaussian elimination on it yields
\[
\left(
\arraycolsep=2pt
\begin{array}{cccc}
{I} & {D} & {E} & {0} \\
{0} & {I} & {I} & {I} \\
{0} & {0} & {I} & {A} \\
{0} & {0} & {0} & C - B + BA \\ 
\end{array}\right).
\]
By the previous equations,
$C - B + BA = EB - BE$, and hence
\begin{equation}\label{eqn:noncom}
{EB}\not={BE},
\end{equation}
which is possible only if $\ell>1$.  

Clearly, the submatrix corresponding to the set $0347$
has rank $3 \ell$ if and only if $C = A$.
But $C \ne A$ because, otherwise, $B=E^{-1}-I$ 
by~\eqref{pmats:1457} and~\eqref{pmats:2467},
and then $EB = BE$, a contradiction with~\eqref{eqn:noncom}. 
As a consequence, $P_1$ and $P''_2$ do not admit
any $(\field,\ell)$-linear representation.

Similarly, the submatrix corresponding to $1256$
has rank $3\ell$ if and only if $D = E$.
We claim that this is impossible and, as a consequence,
$P'_2$ is not a folded linear matroid.
Indeed, if $D = E$, and since $I-E=A$ by~\eqref{pmats:2467} and thus invertible, then 
$B = (I - E)^{-1} E$ by~\eqref{pmats:0456} and 
\[
(I-E) EB = (I-E) E (I-E)^{-1} E =
E (I-E) (I-E)^{-1} E= E^2 = (I-E) BE,
\]
which is a contradiction with~\eqref{eqn:noncom}.
%$B^{-1} = E^{-1}(I - E)$ by~\eqref{pmats:0456}, and by~\eqref{eqn:noncom}, we have
%\[
%B^{-1}E\neq EB^{-1}\implies E^{-1}(I-E)E\neq EE^{-1}(I-E)\implies  I-E\neq I-E\,,
%\] % E^{-1}E(I-E)\neq I-E\implies

Since both $1256$ and $0347$
are bases of $P_3$, it is still possible to find an 
$(\field,\ell)$-linear representation for it.
If there exists such a representation, 
then the matrices corresponding to
$0347$ and $1256$ have full rank, 
and hence the matrices ${B}-{E^{-1}+I}$ and ${E}-{(I-E)B}$ are invertible.
After substituting  ${A}$, ${C}$, and ${D}$ in~\eqref{pmats:rep} 
according to~\eqref{pmats:2467}, \eqref{pmats:1457} and~\eqref{pmats:0456}, 
the following plausible $(\field,\ell)$-linear representation 
for $P_3$ is obtained
\begin{equation}
\left(\label{m:p3}
\arraycolsep=4pt
\begin{array}{cccc|cccc}
{I} & {0} & {0} & {0} & {0} & {I} & {I} & {I} \\
{0} & {I} & {0} & {0} & {I} & {0} & {I} & {I-E} \\
{0} & {0} & {I} & {0} & {I} & {B} & {0} & {EB} \\        
{0} & {0} & {0} & {I} & {I} & {(I-E)B} & {E} & {0} \\        
\end{array} \right).
\end{equation}
As a matter of fact, if we take
$$
B = \left(
\begin{array}{cc}
1 & 1 \\
1 & 0 \\
\end{array} \right)\text{ and   }
E = \left(
\begin{array}{cc}
0 & 2 \\
2 & 0 \\
\end{array} \right)$$ 
%$$
%\left(
%\begin{array}{cccc}
%0 & -1 & 0 &0  \\
%1 & 0 & 0 &0  \\
%0 & 0 & 0 & -1 \\        
%0 & 0 & 1 & 0  \\        
%\end{array} \right)\text{ and   }
%\left(
%\begin{array}{cccc}
%0 & 0 & -1 &0  \\
%0 & 0 & 0 & 1  \\
%1 & 0 & 0 & 0 \\        
%0 & -1 & 0 & 0  \\        
%\end{array} \right),$$ 
it can be checked that it results in a  
$(GF(5),2)$-linear representation for that matroid. 
%Oriol: We verified this using Sage; the code is given in Appendix \ref{app:codeP3}.\anote{Needs verification}

We next prove in a similar fashion 
that $L_8'$ is also a folded linear matroid.
If this is the case, by 
Lemmas~\ref{lemma:invblks} and~\ref{ob:representation}, there exists an $(\field,\ell)$-linear representation of the form
\[
\left(
\arraycolsep=4pt
\begin{array}{cccc|cccc}
{I} & {0} & {0} & {0} & {I} & {I} & {0} & {I} \\
{0} & {I} & {0} & {0} & {D} & {C} & {I} & {A} \\
{0} & {0} & {I} & {0} & {E} & {I} & {I} & {B} \\        
{0} & {0} & {0} & {I} & {F} & {G} & {I} & {0} \\
\end{array}
\right).
\]
Proceeding in the same way as before, 
from the circuit-hyperplanes 
$0156$, $0246$, $1357$, $2347$, and $3456$
we can conclude that
\[{G}={I},\quad {F}={D},\quad {B}={I},\quad 
{A}={D}, \text{ and}\quad {C}={I}-{E}+{D}.\]
%\begin{align}
%\label{l8:0156} {G}&={I}\\
%\label{l8:0246} {F}&={D}\\
%\label{l8:1357} {B}&={I}\\
%\label{l8:2347} {A}&={D}\\
%\label{l8:3456} {C}&={I}-{E}+{D}
%\end{align}
Since $0457$ is a basis, the corresponding 
submatrix
\[
\left(
\arraycolsep=2pt
\begin{array}{cccc}
{I} & {I} & {I} & {I} \\
{0} & {D} & {C} & {A} \\
{0} & {E} & {I} & {B} \\
{0} & {F} & {G} & {0} \\
\end{array}\right)
=
\left(
\arraycolsep=2pt
\begin{array}{cccc}
{I} & {I} & {I} & {I} \\
{0} & {D} & {I-E+D} & {D} \\
{0} & {E} & {I} & {I} \\
{0} & {D} & {I} & {0} \\
\end{array}\right)
\]
has full rank.
By Gausian elimination, we obtain
\[
\qquad
\left(
\arraycolsep=4pt
\begin{array}{cccc}
{I} & {I} & {I} & {I} \\
{0} & {I} & {D}^{-1} & {0} \\
{0} & {0} & {I}-{E}{D}^{-1} & {I} \\
{0} & {0} & {D}{E} {D}^{-1} -{E}& {0}
\end{array}\right),
\]
hence ${D}{E}{D}^{-1}-{E}$ has full rank.
In particular, this implies that $\ell>1$.
In conclusion, if $L_8'$ is a folded linear matroid, 
it admits an $(\field,\ell)$-linear representation of the form
\begin{equation}\label{m:l8'}
\left(
\arraycolsep=4pt
\begin{array}{cccc|cccc}
{I} & {0} & {0} & {0} & {I} & {I} & {0} & {I} \\
{0} & {I} & {0} & {0} & {D} & {I-E+D} & {I} & {D} \\
{0} & {0} & {I} & {0} & {E} & {I} & {I} & {I} \\        
{0} & {0} & {0} & {I} & {D} & {I} & {I} & {0}\\
\end{array}
\right)
\end{equation}
with $DE\neq ED$ and $I-E+D$ invertible. 
Take $i$, with $i^2 = -1$.
The choice 
\[
{D}=\left(
\begin{array}{cc}
0 & -1 \\
1 & 0 \\
\end{array} \right)
\text{ and   }
{E}=\left(
\begin{array}{cc}
i &  0 \\
0 & -i \\
\end{array} \right)
\]
does result in a 
$(GF(5^2),2)$-linear representation of $L'_8$.
This can be checked by using a computer.
%Oriol: We verified this using Sage; the code is given in Appendix \ref{app:codeL8}.
\end{proof}

\subsection{Algebraic Matroids and Skew-Field Representable Matroids}

%%%%%%%%%%%%%%%%%%%%%%%%%%%%%%%%
%% Alegbraic representations? %%
%%%%%%%%%%%%%%%%%%%%%%%%%%%%%%%%

There exist folded linear matroids that are not algebraic~\cite{Ben16},
but none on $8$ points.

\begin{proposition}\label{p:flalg}
Every folded linear matroid on $8$ points is algebraic.
\end{proposition}

\begin{proof}
Since linear matroids are algebraic, 
we only need to consider $P_3$ and $L'_8$. 
Both are algebraic over all fields 
with finite characteristic~\cite[Example~35]{BCD18}.
The result for $P_3$ was first proved by Lindstr\"om~\cite{Lindstrom1986}.
\end{proof}

%\subsection{Linear Representations over Skew-Fields}
%%%%%%%%%%%%%%%%%%%%%%%%%%%%%%
%% Non-commutative matroids %%
%%%%%%%%%%%%%%%%%%%%%%%%%%%%%%

The notion of linear representations of matroids
over fields can be extended 
to linear representations over skew-fields. 
Matroids that admit such a representation are said to be
{\em linearly representable over a skew-field}, 
or {\em skew-field representable} for short. 
The relation between skew-field representable matroids 
and folded linear matroids has been studied in~\cite{PeZw13,Ver15}. 
It is known that there exist folded linear matroids
that are not representable over any skew-field~\cite{PeZw13}. 
In the other direction, some connections have been made in~\cite{Ver15}.
%, but it is still unknown if every skew-field representable matroid is a folded linear matroid. 
We found that, for matroids with at most $8$ points, 
these two classes of matroids coincide. 

\begin{proposition}\label{cor:sf8}
A matroid on at most $8$ points is skew-field representable 
if and only if it is a folded linear matroid.
\end{proposition}

\begin{proof}
Every linearly representable matroid is also skew-field representable. 
Skew-field representable matroids are CI-compliant,
so the 39 non-Ingleton compliant matroids discussed above
are not representable over skew-fields.
The techniques in Section~\ref{pt:folded}
can also be adapted to representations over skew-fields.
In particular, one can prove in that way that
$P_1$, $P'_2$ and $P''_2$ are not skew-field representable. 
Moreover, the matrix~\eqref{m:l8'} provides a
representation of $L'_8$ over the quaternion division ring
$\Real(i,j,k)$ by taking $E=i$ and $D=j$.  
A representation of $P_3$ over the quaternion 
division ring is obtained from the matrix~\eqref{m:p3} 
by taking $B=k$ and $E=j$.
\end{proof}

%%%%%%%%%%%%%%%%%%%%%%%
%% Discussion on diagram goes here %%
%%%%%%%%%%%%%%%%%%%%%%%%%%

\begin{remark}\label{rem:P1andothers}
The only matroids on 8 points for which it is not known whether they are algebraic, almost entropic, or entropic are $P_1$, $P_2'$, and $P_2''$. 
\end{remark}
%Though we know from Theorem~\ref{thm:multilinrep} that the $P_3$ matroid is $(\field,\ell)$-linearly representable, we do not know if it is algebraic or almost entropic. We also do not know whether the other four Ingleton-compliant non-linearly-representable 8-point matroids are algebraic or almost entropic. 
%\anote{add discussion}
%\anote{Missing, e.g., algebraic (Lindstrom result, what about $L'_8$? Can we say something on $P_1$, $P'_2$, or $P''_2$?), non-commutative (Ingleton's result on $P_3$ and $P_1,P'_2,P''_2$, and new result on $L'_8$. Is this a full characterization?), entropic (Can we say something on $P_1$, $P'_2$, or $P''_2$??) }  

We can summarise the current classification of matroids on 8 points as follows. There are 44 matroids that are not linear (Section~\ref{pt:the39}) and, among them, exactly two are folded linear (Theorem~\ref{thm:multilinrep}). Also, on 8 points, a matroid is skew-field representable if and only if it is a folded linear matroid (Proposition~\ref{cor:sf8}), and the folded linear ones are algebraic (Proposition~\ref{p:flalg}). There are three matroids on 8 points for which it is not known whether they are algebraic, almost entropic, or entropic (Remark~\ref{rem:P1andothers}). A classification of these three matroids will conclude the characterization of algebraic, entropic, and almost entropic matroids on 8 points.

\subsection{Exploring Larger Matroids}\label{subs:9,10-ptMat}
%\anote{I think the Tic-Tac-Toe section belongs as a subsubsection here}

By taking into account the results in~\cite{DFZ09}
about linear rank inequalities derived 
from the common information property,
one may expect that there are Ingleton-compliant matroids
that are not CI-compliant.   
As a consequence of the results
in Sections~\ref{pt:the39} and~\ref{pt:folded},
a matroid on $8$ points is $1$-CI-compliant
if and only if it is Ingleton-compliant.  
Mayhew and Royle~\cite{MaRo08} found out that every matroid 
on $9$ points that is not Ingleton-compliant
contains a minor on $8$ points with the same property.
By solving Linear Programming Problem~\ref{ci4rep} for
many matroids on $9$ points from the database~\cite{RoMaDatabase},
we found $171$ sparse paving matroids of rank $5$ on $9$ points
that are Ingleton-compliant but not CI-compliant. All 171 matroids are listed in Table~\ref{tab:ICnoCI9point}.

One of those examples is the tic-tac-toe matroid, 
which is described in Section~\ref{sec:SSSforMatrPort}.
It was shown to be non-linearly representable
by Alfter and Hochst\"attler~\cite{AlHo95}.
Actually, they proved that it does not satisfy the so-called
\emph{generalized Euclidean intersection property},
and the same proof can be used to show that it is not CI-compliant.
It is not known whether the tic-tac-toe matroid is algebraic or not.
By solving Linear Programming Problem~\ref{ak4almentr}, 
we checked that it is $1$-AK-compliant.
We did not find among the other $170$ examples any matroid
that is not $1$-AK-compliant but, due to computational limitations,
our exploration was incomplete.
Of course, the dual matroids of those $171$ matroids are not folded linear.
Nevertheless, we checked that they are $1$-CI-compliant
and hence, by Proposition~\ref{st:CI2AK}, also $1$-AK-compliant.
In addition, we note that, using different techniques, 62 of them were
found to be non-algebraic matroids by Bollen~\cite{Bollen18}.

\begin{figure}
\begin{longtable}{|c|c|c|c|c|c|}
	\hline
	264950 & 265553 & 268475 & 275391 & 282271 & 304085      \\ \hline
	264955 & 265555 & 268476 & 275394 & 282272 & 306452      \\ \hline
	264956 & 265556 & 268477 & 275398 & 283581 & 308279      \\ \hline
	264978 & 265601 & 268486 & 275399 & 283624 & 308280      \\ \hline
	264984 & 265602 & 268611 & 275410 & 283626 & 308285      \\ \hline
	264994 & 265622 & 268613 & 275411 & 283630 & 308381      \\ \hline
	265008 & 265623 & 268765 & 275416 & 283631 & 308385      \\ \hline
	265012 & 265696 & 268774 & 275417 & 283632 & 308386      \\ \hline
	265014 & 265715 & 268805 & 276341 & 291383 & 319504      \\ \hline
	265018 & 265760 & 268958 & 276430 & 292609 & 320838      \\ \hline
	265020 & 266399 & 268961 & 276671 & 293346 & 327043      \\ \hline
	265023 & 266923 & 269060 & 276792 & 293347 & 327134      \\ \hline
	265026 & 266948 & 269061 & 277240 & 293361 & 327157      \\ \hline
	265028 & 267669 & 269062 & 277656 & 294990 & 328810      \\ \hline
	265129 & 267671 & 269550 & 277673 & 295231 & 328817      \\ \hline
	265237 & 267672 & 269551 & 280230 & 299715 & 328818      \\ \hline
	265262 & 267675 & 269558 & 280241 & 299721 & 328917      \\ \hline
	265270 & 267678 & 269559 & 280246 & 300609 & 328928      \\ \hline
	265389 & 267871 & 269704 & 280249 & 300831 & 328941      \\ \hline
	265421 & 267897 & 269824 & 280253 & 301018 & 335557      \\ \hline
	265422 & 267946 & 269895 & 280254 & 303086 & 335558      \\ \hline
	265423 & 268016 & 270130 & 280733 & 303094 & 350495      \\ \hline
	265424 & 268017 & 270133 & 280891 & 303095 & 351377      \\ \hline
	265437 & 268018 & 273139 & 281004 & 303158 & 351471      \\ \hline
	265465 & 268099 & 273141 & 281568 & 303165 & 351483      \\ \hline
	265468 & 268115 & 273582 & 281572 & 303175 & tic-tac-toe \\ \hline
	265547 & 268120 & 274066 & 281581 & 304062 & \phantom{351483} \\ \hline
	265551 & 268272 & 274247 & 281794 & 304066 & \phantom{351483} \\ \hline
	265552 & 268474 & 275082 & 282270 & 304067 & \phantom{351483} \\ \hline
	\caption{Ingleton-compliant non-CI matroids with 9 Points}
	\label{tab:ICnoCI9point}\\
\end{longtable}
\end{figure}

\section{Secret Sharing for Matroid Ports}
\label{sec:SSSforMatrPort}

Consider a finite set of players $P$,
a special player $p_o \notin P$ and $Q = P p_o$.
For a polymatroid $(Q,f)$, we notate
$\Gamma_o(f)$ for its port at $p_o$ and   
$\sigma(f) = \max_{x \in P} f(x)/f(p_o)$.
Let $\Gamma$ be a connected access structure on the set $P$.
Then the parameters 
$\sigma(\Gamma)$ and $\lambda(\Gamma)$
introduced in Section~\ref{pt:presss}
and the optimal value $\kappa(\Gamma)$ 
of Linear Programming Problem~\ref{tab:LPkap1}
are characterized as follows.
\begin{itemize}
\item
$\kappa(\Gamma) = 
\min \{\sigma(f) \,:\, (Q,f) 
\mbox{ is a polymatroid with } \Gamma = \Gamma_o(f) \}$.
\item
$\sigma(\Gamma) = 
\inf \{\sigma(f) \,:\, (Q,f) 
\mbox{ is an entropic  polymatroid with } \Gamma = \Gamma_o(f) \}$.
\item
$\lambda(\Gamma) = 
\inf \{\sigma(f) \,:\, (Q,f) 
\mbox{ is a linear polymatroid with } \Gamma = \Gamma_o(f) \}$.
\end{itemize}
The following parameter has been 
recently introduced by Csirmaz~\cite{Csi19}.
\begin{itemize}
\item
$\overline{\sigma}(\Gamma) = 
\min \{\sigma(f) \,:\, (Q,f) 
\mbox{ is an almost entropic  polymatroid with } \Gamma = \Gamma_o(f) \}$.
\end{itemize}
Clearly, 
$1 \le \kappa(\Gamma) 
\le \overline{\sigma}(\Gamma)
\le \sigma(\Gamma)
\le \lambda(\Gamma)$.
Moreover, $\Gamma$ is a matroid port
if and only if $\kappa(\Gamma) = 1$,
and this is equivalent to 
$\kappa(\Gamma) < 3/2$~\cite[Theorem~4.4]{MaPa10}.
An access structure admits an
ideal secret sharing scheme
if and only if it is the port of an entropic matroid.
Besides, $\overline{\sigma}(\Gamma) = 1$
if and only if $\Gamma$ is the port of an almost entropic matroid.
The parameters $\kappa$ and $\lambda$ are invariant  
by duality, that is, $\kappa(\Gamma^*) = \kappa(\Gamma)$
and $\lambda(\Gamma^*) = \lambda(\Gamma)$
for every access structure $\Gamma$.
By the recent results in~\cite{Csi19,Kac18},
this is not the case for the parameter $\overline{\sigma}$.
If the access structure $\Gamma'$ is a minor of $\Gamma$, 
then $\kappa(\Gamma') \le \kappa(\Gamma)$, 
$\lambda(\Gamma') \le \lambda(\Gamma)$, and also
$\overline{\sigma}(\Gamma') \le \overline{\sigma}(\Gamma)$.

By using the techniques described in 
Section~\ref{sec:SecShaSch}, new lower bounds on 
$\overline{\sigma}(\Gamma)$ and $\lambda(\Gamma)$
were obtained in~\cite{FKMP20} for several access structures
including the ports of the matroids
$AG(3,2)'$, $F_8$, $Q_8$, and $V_8$.
Moreover, the bounds on $\lambda(\Gamma)$
for the ports of  $Q_8$ and $V_8$ are tight~\cite{FKMP20}.
Subsequently, an improved lower bound 
on $\overline{\sigma}(\Gamma)$ for a 
port of the Vamos matroid $V_8$ was
obtained in~\cite{GuRo2019} by
using the copy lemma instead of 
the Ahlswede--K\"orner lemma.

In this work, we continued the 
search for lower bounds for matroid ports
by using those methods, which, of course, provide relevant lower bounds
only when applied to matroids that are not CI-compliant.
We began by exploring the ports of the 39 matroids 
on $8$ points that are not Ingleton-compliant
and we found out that all of them satisfy 
$\lambda(\Gamma) \ge 4/3$  and
$\overline{\sigma}(\Gamma) \ge 9/8$.
A more general result is obtained by combining 
our bounds with Corollary~\ref{st:minIng}.

\begin{theorem}\label{thm:l&s4nonIC}
If a sparse paving matroid is not Ingleton-compliant,
then at least eight of its ports satisfy 
$\lambda(\Gamma)\geq 4/3$ and
$\overline{\sigma}(\Gamma)\geq 9/8$.
\end{theorem}

\begin{proof}
Let $M = (Q,r)$ be a sparse paving matroid 
that is not Ingleton-compliant. 
By Corollary~\ref{st:minIng}, it has a minor $M' = (Q',r')$ 
with $|Q'| = 8$ that is not Ingleton-compliant.
Hence $M'$ is one of the 39 matroids 
on $8$ points that are not Ingleton-compliant.
For every $p_o \in Q' \subseteq Q$, the port $\Gamma'$ of $M'$ at $p_o$ 
is a minor of the port $\Gamma$ of $M$ at $p_o$.
Therefore, 
$\lambda(\Gamma) \ge \lambda(\Gamma')\ge 4/3$
and $\overline{\sigma}(\Gamma) \ge\overline{\sigma}(\Gamma') \ge 9/8$.
\end{proof}

Better lower bounds on $\overline{\sigma}(\Gamma)$
have been obtained for some of those $39$ matroids,
which are presented in Table \ref{table:8ptstable}.
The names or numbers of the matroids are as they appear in \cite{MaRo08}, and in the database~\cite{RoMaDatabase}.
%which coincides with the SAGE database of matroids.

\begin{figure}
\begin{longtable}{@{}ccc@{}}
	\toprule
	\label{table:8ptstable}
	Matroid          & Port                    & Improved bound on $\overline{\sigma}(\Acc)$ \\* \midrule
	\endfirsthead
	\multicolumn{3}{c}%
	{{\bfseries Table \thetable\ continued from previous page}} \\
	\toprule
	Matroid          & Port                    & Improved bound on $\overline{\sigma}(\Acc)$ \\* \midrule
	\endhead
	\bottomrule
	\endfoot
	\endlastfoot
	1490           & 0, 2, 3, 4, 5, 6         & 8/7               \\
	1491           & 0, 3, 7                  & 33/29             \\
	1491           & 2, 4, 5, 6               & 8/7               \\
	1492           & 0, 1, 2, 3, 4, 5, 6, 7   & 49/43             \\
	1494           & 3, 4, 5, 6               & 33/29             \\
	1499           & 0, 2, 3, 4, 5, 6         & 8/7               \\
	1500           & 0, 2, 3, 4, 5, 6         & 8/7               \\
	1501           & 0, 1, 2, 3, 6, 7         & 33/29   \\
	1501           & 4, 5                     & 8/7   \\
	1502           & 5, 6                     & 8/7   \\
	1502           & 2, 3, 4, 7               & 33/29   \\
	1508           & 3, 4, 5, 6               & 33/29   \\
	1509           & 3, 4, 5, 6               & 33/29   \\
	1510           & 3, 4, 5, 6               & 33/29   \\
	1518           & 3, 4, 5, 6               & 33/29   \\
	1520           & 2, 3, 4, 7               & 33/29   \\
	1524           & 3, 4, 5, 6               & 33/29   \\
	1525           & 0, 2, 4, 5               & 33/29   \\
	1525           & 3, 6                     & 8/7   \\
	1526           & 0, 2, 3, 4, 5, 6         & 8/7   \\
	1527           & 0, 2, 4, 5               & 33/29   \\
	1528           & 0, 2, 3, 6               & 8/7   \\
	1529           & 1, 4, 5, 7               & 33/29   \\
	1531           & 2, 5, 6, 7               & 33/29   \\
	1532           & 4, 7                     & 8/7   \\
	1532           & 0, 1, 2, 3, 5, 6         & 33/29   \\
	1549           & 3, 4, 5, 6               & 33/29   \\
	1568           & 3, 4, 5, 6               & 33/29   \\
	1572           & 2, 3, 4, 7               & 33/29   \\
	1576           & 3, 4, 5, 6               & 33/29   \\
	1578           & 3, 4, 5, 6               & 33/29   \\
	1579           & 0, 2, 4, 5               & 33/29   \\
	1579           & 3, 6                     & 8/7   \\
	1580           & 0, 2, 3, 6               & 33/29   \\
	1641           & 3, 4, 5, 6               & 33/29   \\
	1646           & 2, 5, 6, 7               & 33/29   \\
	1654           & 3, 4, 5, 6               & 33/29   \\
	1656           & 0, 2, 3, 6               & 33/29   \\
	1657           & 0, 2, 3, 6               & 33/29   \\
	1660           & 0, 2, 3, 6               & 33/29   \\
	$AG(3,2)'$ & 1, 3, 5, 7               & 49/43             \\
	$F_8$        & 1, 7                     & 8/7               \\
	$F_8$        & 3, 4, 5, 6               & 33/29             \\
	$Q_8$        & 1, 4, 6, 7               & 49/43             \\
	$V_8^+$      & 0, 2, 3, 6                    & 33/29       \\
	$V_8$        & 2, 3, 6, 7               & \phantom{/}33/29$^\dagger$ \\* \bottomrule
	\caption{Bounds on ports of matroids on 8 points. $^\dagger$Improved in~\cite{GuRo2019} }
\end{longtable}
\end{figure}

%\subsection{Ports of Matroids with 9 points}\label{9ptmatroids}
%Results on non-Ingleton on 9 points removed
%Alongside matroids with 8 points, we also studied some 9-point matroids from the matroid database presented by Mayhew and Royle. While most of the matroids we checked were CI-compliant matroids, some of them were found to be non-CI-compliant. For majority of these non-CI matroids, we obtained the bound $4/3\leq\lambda(\Acc)$ for their ports, thereby coinciding with the result we got for the 8-point matroids we studied. However, we also got the bounds $7/6$ and $6/5$ for a few of the ports. A more detailed description of the steps we took in approaching these matroids and the results we obtained are presented in Appendix~\ref{apndx:9pts}. It is perhaps interesting to mention here that the ports of the matroids considered that gave a bound different from $4/3$ have one of the non-Ingleton-compliant 8-point matroids as a minor. 

We also applied the linear programs in Section~\ref{sec:SecShaSch}
to the ports of matroids 265389, 265421, 265468, 265551, 265556, 
265622, and the tic-tac-toe matroid; all Ingleton-compliant but non-CI-compliant matroids on nine points. For all of them, we obtained the lower bound $\lambda(\Gamma) \ge 6/5$.  
We were not able to find any non-trivial 
bound on $\overline{\sigma}(\Gamma)$.

By presenting a suitable linear secret sharing scheme,
we prove next that the bound $\lambda(\Gamma) \ge 6/5$
is tight for at least one of the ports of the tic-tac-toe matroid.
Take $Q = \{0,1,2\} \times \{0,1,2\}$
and, for every $(a,b) \in Q$, the $5$-element set
\[
C_{ab} = \{(i,j) \in Q \,:\, i = a \mbox{ or } j = b\}.
\]
We introduce several sparse paving matroids
with ground set $Q$ and rank $5$.
We call $M_o$ the one whose 
circuit-hyperplanes are all sets $C_{ab}$.
The \emph{tic-tac-toe matroid} $M$
is obtained from $M_o$ by relaxing the circuit $C_{11}$.
Finally, for every $(a,b) \ne (1,1)$, let $M_{ab}$
be the matroid that is obtained from the tic-tac-toe matroid
by relaxing the circuit $C_{ab}$.
Clearly, every matroid $M_{ab}$ is isomorphic
to either $M_{00}$ or $M_{01}$.
The matroids $M_o$ and $M_{ab}$ with $(a,b) \ne (1,1)$ 
are representable over every large enough field.
We skip the proof of this fact, 
but we present $\field_{11}$-linear representations
for $M_o$, $M_{00}$, and $M_{01}$,
which are given, respectively,
by the following matrices, whose columns are indexed as
$(0,0), (0,1), (0,2), (1,0), (1,1), (1,2), (2,0), (2,1), (2,2)$.
%%%%%%%%%%%%%%%%%%
\[
\left(
\begin{array}{ccccccccc}
1 & 0 & 1 & 1 & 1 & 0 & 0 & 1 & 1 \\
1 & 1 & 0 & 1 & 1 & 0 & 0 & 0 & 0 \\
0 & 0 & 0 & 0 & 1 & 1 & 0 & 1 & 1 \\
0 & 0 & 0 & 1 & 1 & 0 & 1 & 1 & 0 \\
0 & 6 & 0 & 1 & 0 & 4 & 0 & 2 & 3 \\
\end{array}
\right)\qquad
\left(
\begin{array}{ccccccccc}
1 & 1 & 1 & 1 & 1 & 0 & 0 & 1 & 1 \\
0 & 0 & 0 & 5 & 0 & 1 & 1 & 0 & 10 \\
1 & 0 & 8 & 1 & 0 & 3 & 0 & 0 & 0 \\
0 & 0 & 0 & 1 & 1 & 0 & 6 & 7 & 0 \\
1 & 5 & 0 & 1 & 1 & 0 & 0 & 0 & 0 \\
\end{array}
\right)
\]
%%%%%%%%%%%%%%%%%%%
\[
\left(
\begin{array}{ccccccccc}
1 & 0 & 1 & 1 & 1 & 0 & 0 & 7 & 1 \\
1 & 1 & 0 & 1 & 1 & 0 & 0 & 0 & 0 \\
0 & 0 & 0 & 0 & 1 & 1 & 0 & 7 & 1 \\
0 & 0 & 0 & 1 & 1 & 0 & 1 & 1 & 0 \\
9 & 0 & 7 & 6 & 0 & 3 & 6 & 0 & 6 \\
\end{array}
\right).
\]
%%%%%%%%%%%%%%%%%%%%%
Let $\Gamma$ be the port of the tic-tac-toe
matroid $M$ at $p_o = (0,0)$.
Let $\Gamma_{11}$ be the port
of $M_o$ at $p_o$ and, for $(a,b) \ne (1,1)$,
let $\Gamma_{ab}$ be the port of $M_{ab}$ at $p_o$.
Since they are ports of $\field_{11}$-linear matroids, 
each of the nine access structures $\Gamma_{ab}$ admits
an ideal $\field_{11}$-linear secret sharing scheme.
Every qualified set of $\Gamma$
is qualified in at least five of the 
six access structures
$\Gamma_{11}$, 
$\Gamma_{00}$, 
$\Gamma_{01}$, 
$\Gamma_{02}$, 
$\Gamma_{10}$, and 
$\Gamma_{20}$.
In addition, the unqualified sets of $\Gamma$
are also unqualified in those six access structures.
Therefore, by combining the ideal linear secret sharing schemes
for those six access structures
in a $\lambda$-decomposition with $\lambda = 5$,
we obtain a linear secret sharing scheme for~$\Gamma$
with information ratio $6/5$.
The reader is referred to~\cite{Pad12,Sti94}
for more information about $\lambda$-decompositions.
 
\paragraph{Acknowldegements:} We thank Dillon Mayhew and Gordon F. Royle for helpful suggestions and also for providing us the matroid database~\cite{RoMaDatabase}. We thank Guus P. Bollen for his helpful suggestions on algebraic matroids.


\begin{thebibliography}{11}
\bibitem{Matus17}
	Mat\'u\v{s}, F.:
	Algebraic matroids are almost entropic. 
	To appear in Proceedings of the AMS
\end{thebibliography}
\end{document}